\documentclass[12pt]{amsart}
\usepackage{a4}
\usepackage{amsthm, amsfonts, amssymb,latexsym}
\usepackage{enumerate, color}
\usepackage[english]{babel}%[english]
\usepackage[latin1]{inputenc}
\usepackage{tikz}
\usepackage{float}

\newtheorem{lemma}{Lemma}[section]

\newtheorem{theorem}{Theorem}[section]

\numberwithin{equation}{section}

\newcommand{\beq}[1]{\begin{equation}\label{#1}}
\newcommand{\eeq}{\end{equation}}

\title[A new expander and improved bounds for $A(A+A)$]{A new expander and improved bounds for $A(A+A)$}

\author[ O. Roche-Newton]{Oliver Roche-Newton}

\address{O. Roche-Newton: School of Mathematics and Statistics, Wuhan University, Wuhan, Hubei Province, P.R.China. 430072 }
\email{o.rochenewton@gmail.com }

\begin{document}

\begin{abstract}
The main result in this paper concerns a new five-variable expander. It is proven that for any finite set of real numbers $A$,
$$|\{(a_1+a_2+a_3+a_4)^2+\log a_5 :a_1,a_2,a_3,a_4,a_5 \in A \}| \gg \frac{|A|^2}{\log |A|}.$$
This bound is optimal, up to logarithmic factors. The paper also gives new lower bounds for $|A(A-A)|$ and $|A(A+A)|$, improving on results from \cite{MORNS}. The new bounds are
$$|A(A-A)| \gtrapprox |A|^{3/2+\frac{1}{34}}$$
and
$$|A(A+A)| \gtrapprox |A|^{3/2+\frac{5}{242}}.$$
\end{abstract} 
\maketitle
\section{Introduction}

Throughout this paper, the standard notation
$\ll,\gg$ is applied to positive quantities in the usual way. Saying $X\gg Y$ means that $X\geq cY$, for some absolute constant $c>0$. All logarithms in the paper are base $2$. We use the symbols $\lessapprox, \gtrapprox$ to suppress both constant and logarithmic factors. To be precise, we write $X \gtrapprox Y$ if there is some absolute constant $c>0$ such that $X \gg Y/(\log X)^c$.

This paper is concerned with a particular variation of the sum-product problem. A fundamental idea in sum-product theory is that a finite set $A$ in a field $\mathbb F$ cannot be highly structured in both an additive and multiplicative sense. This is a guiding principle behind the Erd\H{o}s-Szemer\'{e}di sum-product conjecture, which states that for any finite set $A \subset \mathbb Z$ and all $\epsilon >0$,
$$\max \{ |A+A|, |AA| \} \geq c_{\epsilon}|A|^{2-\epsilon},$$
where $A+A:= \{a+b:a,b \in A \}$ is the \textit{sum set}, and the \textit{product set} $AA$ is defined analagously. The conjecture remains wide open. See the recent work of Konyagin and Shkredov \cite{KS2} for the most up to date bounds, and the references within for more background on this famous problem. See also chapter 8 of \cite{TV} for a more detailed, although now slightly outdated, introduction to the sum-product problem.

By the same principle, it is typically expected that a set which is defined by a combination of additive and multiplicative operations on elements of an input set $A$ will always be large compared to $A$. For example, it follows from an ingenious geometric argument of Ungar \cite{ungar} that for any finite set $A \subset \mathbb R$,
\begin{equation}
\left| \frac{A-A}{A-A} \right | \geq |A|^2 -2,
\label{ungarbound}
\end{equation}
where
$$\frac{A-A}{A-A}:= \left \{ \frac{a-b}{c-d}:a,b,c,d \in A, c\neq d \right \}.$$

We say that a function\footnote{We usually take $D=\mathbb R^d$, but we use a general $D\subset \mathbb R^d$ in this definition to avoid the possibility of dividing by zero.} $f: D \rightarrow \mathbb R$ is a \textit{$d$-variable expander} if $D \subset \mathbb R^d$ and it is true that there is some $\epsilon >0$ such that for any finite $A \subset \mathbb R$
$$|\{ f(a_1,\dots,a_d) : a_i \in A \}| \gg |A|^{1+\epsilon}.$$
So, inequality \eqref{ungarbound} shows that the function $f(x_1,x_2,x_3,x_4)=\frac{x_1-x_2}{x_3-x_4}$ is a $4$-variable expander.

Results on expanders which are tight up to constant and logarithmic factors are relatively rare. For $3$ variables, the only such result is due to Jones \cite{TJthesis}, who proved that
\begin{equation}
\left | \left \{ \frac{a-c}{a-b}:a,b,c \in A \right \} \right | \gg \frac{|A|^2}{\log |A|}.
\label{TJ}
\end{equation}
A slightly different proof of this inequality can also be found in \cite{SOCG}. For $4$ variables, as well as the aforementioned result \eqref{ungarbound}, it is known that
\begin{equation}
|\{(a-b)^2+(c-d)^2:a,b,c,d \in A \}| \gg \frac{|A|^2}{\log |A|}
\label{GK}
\end{equation}
and
\begin{equation}
|(A+A)(A+A)| \gg \frac{|A|^2}{\log |A|}.
\label{rectangles}
\end{equation}
The results are due to Guth and Katz \cite{GK} and Roche-Newton and Rudnev \cite{rectangles} respectively. A considerably more simple proof of \eqref{rectangles} was given in \cite{SOCG}. It was also established by Balog and Roche-Newton \cite{BORN} that
\begin{equation}
 \left | \frac{A+A}{A+A} \right | \geq 2|A|^2-1.
\label{balog}
\end{equation}
The following $5$-variable expander result was proven by Murphy, Roche-Newton and Shkredov \cite{MORNS}:
\begin{equation}
|A(A+A+A+A)| \gg \frac{|A|^2}{\log |A|}.
\label{a(a+a+a+a)}
\end{equation}
All of these results are optimal up to constant and logarithmic factors, as can be seen by taking $A$ to be an arithmetic progression, and this gives a complete list of the known optimal expander results for $5$ or less variables. In this paper, we add to the list by proving the following theorem:

\begin{theorem}\label{thm:main}
For any finite set of real numbers $A$,
$$|\{(a_1+a_2+a_3+a_4)^2+\log a_5 :a_i \in A \}| \gg \frac {|A|^2}{\log |A|}.$$
\end{theorem}

This gives an optimal bound, up to log factors, for the admittedly curious expander function $$f(a_1,a_2,a_3,a_4,a_5)=(a_1+a_2+a_3+a_4)^2+\log a_5.$$

This paper also considers the more natural expander $f(a,b,c)=a(b-c)$. It is easy to use the Szemer\'{e}di-Trotter Theorem to prove that $|A(A-A)| \gg |A|^{3/2}.$ See \cite[Exercise 8.3.3]{TV} for a similar result. In \cite{MORNS}, this was improved to $|A(A-A)| \gg |A|^{\frac{3}{2}+\frac{1}{112}}$. Here, we improve this further:
\begin{theorem}
For any finite set $A$ of real numbers
$$|A(A-A)| \gtrapprox |A|^{\frac{3}{2}+\frac{1}{34}}.$$
\end{theorem}
Similarly, we prove the following result for the expander $f(a,b,c)=a(b+c)$.
\begin{theorem}
For any finite set $A$ of real numbers
$$|A(A+A)| \gtrapprox |A|^{\frac{3}{2}+\frac{5}{242}}.$$
\end{theorem}
This improves the bound $|A(A+A)| \gg |A|^{\frac{3}{2}+\frac{1}{178}}$ from \cite{MORNS}. The proofs of these two theorems use ideas from a recent paper of Konyagin and Shkredov \cite{KS2} to streamline the original argument by avoiding using the Balog-Szemer\'{e}di-Gowers Theorem.

\section{Notation and preliminary results}

Given finite sets $A,B\subset \mathbb R$, the \textit{additive energy of $A$ and $B$} is the number of solutions to the equation
$$a_1-b_1=a_2-b_2$$
such that $a_1,a_2 \in A$ and $b_1,b_2 \in B$. The additive energy is denoted $E^+(A,B)$. Let 
$$r_{A-B}(x):=|\{(a,b) \in A \times B: a-b=x \}|.$$
Note that $r_{A-B}(x)=|A \cap (B+x)|$. The notation of the representation function $r$ will be used with flexibility throughout this paper, with the information about the kind of representations it counts being contained in a subscipt. For example, $$r_{(A-A)^2+(A-A)^2}(x);=|\{(a_1,a_2,a_3,a_4) \in A^4 : (a_1-a_2)^2+(a_3-a_4)^2=x \}|.$$

Note that
$$E^+(A,B)= \sum_{x \in A-B}r^2_{A-B}(x).$$
The shorthand $E^+(A,A)=E^+(A)$ is used. The notion of energy can be extended to an arbitrary power $k$. We define $E^+_k(A)$ by the formula
$$E^+_k(A)= \sum_{x \in A-A}r^k_{A-A}(x).$$

Similarly, the \textit{multiplicative energy of $A$ and $B$}, denoted $E^*(A,B)$, is the number of solutions to the equation
$$\frac{a_1}{b_1}=\frac{a_2}{b_2},$$
such that $a_1,a_2 \in A$ and $b_1,b_2 \in B$.

The notions of additive and multiplicative energy have been central in the literature on sum-product estimates. For example, the key ingredient in the beautiful work of Solymosi \cite{solymosi}, which until recently held the record for the best known sum-product estimate, is the following bound:

\begin{theorem} \label{solymosi}
For any finite set $A \subset \mathbb R$,
$$E^*(A) \ll |A+A|^2 \log |A|.$$
\end{theorem}

We will also call upon the following result on the relationship between different types of energy:

\begin{lemma}[\cite{Li}, Lemma 2.4 and Lemma 2.5] \label{li} For any finite sets $A,B \subset \mathbb R$,
$$|A|^2 (E_{1.5}^+(A))^2 \leq (E_3^+(A))^{2/3}(E_3^{+}(B))^{1/3}E(A,A-B).$$
\end{lemma}

In \cite{MORNS}, the following lemma played an important role:

\begin{lemma} \label{MORNSlemma}
For any finite sets $A,B,C \in \mathbb R$,
$$E^*(A)|A(B+C)|^2 \gg \frac{|A|^4|B||C|}{\log |A|}.$$
\end{lemma}
The proof in \cite{MORNS} uses only the Szemer\'{e}di-Trotter Theorem, but it can also be proved in a more superficially straightforward way by using ideas from the work of Guth and Katz \cite{GK} on the Erd\H{o}s distinct distance problem. In particular, Lemma \ref{MORNSlemma} follows very easily if we assume the following result from \cite{rectangles}:

\begin{theorem} \label{rectanglesthm} For any finite set $A \subset \mathbb R$, the number of solutions to the equation
$$(a_1 - a_2)(a_3 - a_4)=(a_5-a_6)(a_7-a_8)$$
such that $a_1,\dots,a_8 \in A$ is $O(|A|^6 \log |A|)$.
\end{theorem}
The same result holds with the minus signs changed to plus signs. The proof of Theorem \ref{rectanglesthm} in \cite{rectangles} closely follows the work of Guth and Katz, and is an analogue of the following result from \cite{GK}:

\begin{theorem} \label{GKthm} For any finite set $A \subset \mathbb R$, the number of solutions to the equation
$$(a_1 - a_2)^2+(a_3 - a_4)^2=(a_5-a_6)^2+(a_7-a_8)^2$$
such that $a_1,\dots,a_8 \in A$ is $O(|A|^6 \log |A|)$.
\end{theorem}
Once again, the same bound holds for the equation
$$(a_1 + a_2)^2+(a_3 + a_4)^2=(a_5+a_6)^2+(a_7+a_8)^2.$$
In fact, Theorems \ref{GKthm} and \ref{rectanglesthm} are special cases of more general geometric results which were proved respectively in \cite{GK} and \cite{rectangles}, but here they are stated only in the forms in which they will be used in this paper.

Theorem \ref{GKthm} can be used to prove the following variation of Lemma \ref{MORNSlemma}:
\begin{lemma} \label{newlemma}
For any finite sets $A,B \in \mathbb R$,
$$E^+(A)|\{a+(b_1+b_2)^2:a \in A, b_1,b_2 \in B\}|^2 \gg \frac{|A|^4|B|^2}{\log |B|}.$$
\end{lemma}

\begin{proof} The proof proceeds by the familiar method of double counting the number of solutions to the equation
\begin{equation}
a_1+(b_1+b_2)^2=a_2+(b_3+b_4)^2
\label{soln}
\end{equation}
such that $a_i \in A$ and $b_i \in B$. Let $S$ denote the number of solutions to \eqref{soln} and write
$$A+(B+B)^2:=\{a+(b_1+b_2)^2:a \in A, b_1,b_2 \in B\}.$$
By the Cauchy-Schwarz inequality
$$S \geq \frac{|A|^2|B|^4}{|A+(B+B)^2|}.$$
On the other hand, also by the Cauchy-Schwarz inequality
\begin{align*}
S^2&=\left(\sum_{x}r_{A-A}(x)r_{(B+B)^2-(B+B)^2}(x)\right)^2
\\ &\leq \left(\sum_x r^2_{A-A}(x) \right) \left(\sum_x r_{(B+B)^2-(B+B)^2}^2(x) \right)
\\&=E^+(A) \left(\sum_x r_{(B+B)^2-(B+B)^2}^2(x) \right)
\end{align*}
Theorem \ref{GKthm} tells us that $ \left(\sum_x r_{(B+B)^2-(B+B)^2}^2(x) \right)=O(|B|^6\log |B|)$. Therefore
\begin{align*}
|A|^4|B|^8 &\leq  |A+(B+B)^2|^2 S^2
\\ & \ll  |A+(B+B)^2|^2 E^+(A) |B|^6 \log|B|.
\end{align*}
After rearranging this inequality, we obtain the desired result.\end{proof}

Unfortunately, we are not aware of a proof of Lemma \ref{newlemma} which does not use the deep results from \cite{GK}.

\section{Five variable expander}

It is now straightforward to use the results from the previous section to prove the result on the new five variable expander.

\begin{theorem}
For any finite set $A\subset \mathbb R$
$$|\{(a_1+a_2+a_3+a_4)^2+\log a_5 :a_i \in A \}| \gg \frac{|A|^2}{\log |A|}.$$
\end{theorem}

\begin{proof}
Apply Lemma \ref{newlemma} with $A= \log A$ and $B=A+A$. We have
$$E^+(\log(A)) |\{(a_1+a_2+a_3+a_4)^2+\log a_5 : a_i \in A\}|^2 \gg \frac{|A|^4|A+A|^2}{\log |A|}.$$
Note that $\log a_1 + \log a_2=\log a_3+ \log a_4$ if and only if $a_1a_2=a_3a_4$, and so $E^+(\log(A))=E^*(A)$. We can apply Theorem \ref{solymosi} to deduce that 
$$E^+(\log A) \ll |A+A|^2 \log |A|.$$ 
It then follows that
$$|\{(a_1+a_2+a_3+a_4)^2+\log a_5 : a_i \in A\}|^2 \gg \frac{|A|^4}{\log^2 |A|},$$
which completes the proof.

\end{proof}

\section{Three variable expanders}

In a recent paper of Konyagin and Shkredov \cite{KS2}, a new characteristic for a finite set of real numbers $A$ was considered. Define $d_*(A)$ by the formula
$$d_*(A)= \min_{t>0} \min_{\emptyset \neq Q,R \subset \mathbb R \setminus \{0\}} \frac{|Q|^2|R|^2}{|A|t^3},$$
where the second minimum is taken over all $Q$ and $R$ such that $\max \{|Q|,|R|\} \geq |A|$ and such that for every $a\in A$, the bound $|Q \cap aR^{-1}| \geq t$ holds. Konyagin and Shkredov proved the following lemma:

\begin{lemma}[Lemma 13,\cite{KS2}] \label{KSlemma}
For any $A,B \subset \mathbb R$ and any $\tau \geq 1$,
$$|\{x : r_{A-B}(x) \geq \tau \}| \ll \frac{|A||B|^2}{\tau^3}d_*(A).$$
\end{lemma}

The proof uses the Szemer\'{e}di-Trotter Theorem. Lemma \ref{KSlemma} generalises an earlier result in which the bound
\begin{equation}
|\{x : r_{A-B}(x) \geq \tau \}| \ll \frac{|A||B|^2}{\tau^3}d(A)
\label{KSlemma2}
\end{equation}
was established, where $d(A)=\min_{C \neq \emptyset}\frac{|AC|^2}{|A||C|}.$ See \cite[Lemma 7]{RRS} for a proof. As pointed out in \cite{KS2}, $d_*(A) \leq d(A)$, since for any non empty $C$ we can take $t=|C|$, $Q=AC$ and $R=C^{-1}$ in the definition of $d_*(A)$.

Using the language of \cite{KS2} and \cite{Sh}, we could rephrase Lemma \ref{KSlemma} by saying that $A$ is a Szemer\'{e}di-Trotter set with $O(d_*(A))$.

We can use Lemma \ref{KSlemma} to prove the following lemma. No originality is claimed here - we are essentially copying the arguments from \cite{RRS} and predecessors with the stronger Lemma \ref{KSlemma} in place of the bound \eqref{KSlemma2} - but we include the proof for completeness.

\begin{lemma} \label{A-A} For any finite set $A \subset \mathbb R$,
$$|A-A| \gg \frac{|A|^{8/5}}{d^{3/5}_*(A)\log^{2/5}|A|}.$$
\end{lemma}

\begin{proof} First, we will prove two energy bounds. Note that, by Lemma \ref{KSlemma},
\begin{align} \label{E3}
E_3^+(A)&=\sum_x r_{A-A}^3(x) \nonumber
\\&=\sum_{j\geq 1} \sum_{x:2^{j-1}\leq r_{A-A}(x) < 2^j}r_{A-A}^3(x) \nonumber
\\& \ll |A|^3d_*(A) \log |A|.
\end{align}

Similarly, for any $F \subset \mathbb R$,
\begin{align*}
E^+(A,F)&=\sum_x r_{A-F}^2(x)
\\&=\sum_{x:r_{A-F}<\Delta}r_{A-F}^2(x)+\sum_{j\geq 1} \sum_{x:\Delta2^{j-1} \leq r_{A-F}(x) < \Delta2^j}r_{A-F}^2(x)
\\& \ll \Delta|A||F|+\frac{ |A||F|^2d_*(A)}{\Delta}.
\end{align*}
We choose $\Delta=(|F|d_*(A))^{1/2}$, and thus conclude that
\begin{equation}
E(A,F) \ll |A||F|^{3/2}d_*(A)^{1/2}.
\label{Ebound}
\end{equation}

Now, by H\"{o}lder's inequality,
\begin{align*}
|A|^6 &= \left ( \sum_{x \in A-A} r_{A-A}(x) \right)^3
\\ & \leq \left(\sum_x r_{A-A}^{3/2} \right )^2|A-A|
\\&=(E_{1.5}^+(A))^2|A-A|.
\end{align*}
Finally, applying Lemma \ref{li} as well as inequalities \eqref{E3} and \eqref{Ebound}, we have
\begin{align*}
|A|^8 &\leq E_3^+(A)E^+(A,A-A)|A-A|
\\& \ll |A|^4 |A-A|^{5/2} d_*^{3/2}(A) \log |A|,
\end{align*}
and it follows that
$$ |A-A| \gg \frac{|A|^{8/5}}{d_*^{3/5}(A) \log^{2/5}|A|}.$$

\end{proof}

The following similar result for sum sets follows from a combination of the work in \cite{Sh} and \cite{KS2}:
\begin{lemma} \label{ShSum} For any finite set $A \subset \mathbb R$,
$$|A+A| \gtrapprox \frac{|A|^{58/37}}{d^{21/37}_*(A)}.$$
\end{lemma}

To be more precise, it was proven in \cite{Sh} that if $A$ is a Szemer\'{e}di-Trotter set with $D$, then $|A+A| \gtrapprox \frac{|A|^{58/37}}{D^{21/37}}$, and it was subsequently established in \cite{KS2} that any set $A$ is a Szemer\'{e}di-Trotter set with $O(d_*(A))$.

The methods used in \cite{Sh} are rather different from those used in this paper, and appear to be far from trivial. However, one can obtain a quantitatively weaker bound
\begin{equation}
|A+A| \gg \frac{|A|^{14/9}}{d^{5/9}_*(A)\log^{2/9}|A|}
\label{ShSumEasy}
\end{equation}
with a proof very similar to that of Lemma \ref{A-A} above. To see how this works, one can repeat the arguments from the proof of Theorem 1.2 in \cite{LiORN2}, but using Lemma \ref{KSlemma} in place of Lemma 3.2 from \cite{LiORN2}. This is worth noting, since the proofs of the main results in \cite{KS} and \cite{KS2}, that is the bound
$$\max \{ |A+A|, |AA| \} \gg |A|^{4/3+c},$$
for some $c>0$, both include applications of Lemma \ref{ShSum}. One can also obtain this sum-product estimate, albeit with a smaller value of $c$, by using the bound \eqref{ShSumEasy} instead of Lemma \ref{ShSum}.

The Balog-Szemer\'{e}di-Gowers Theorem tells us that if a set $A$ has large multiplicative energy, then there is a large subset $A'\subset A$ such that $|A'/A'|$ is small. We can then use sum-product theory to deduce that $A'+A'$, and thus also $A+A$, is large. The following result arrives at the same conclusion, but avoids applying the Balog-Szemer\'{e}di-Gowers Theorem, and therefore gives quantitatively better estimates for the problem at hand. The proof uses ideas from \cite{KS2}.

\begin{lemma} \label{BSG2} Let $A \subset \mathbb R$ and suppose that $E^*(A) \geq \frac{|A|^3}{K}$. Then 
$$|A-A| \gtrapprox \frac{|A|^{8/5}}{K^{6/5}}$$
and
$$|A+A| \gtrapprox \frac{|A|^{58/37}}{K^{42/37}}.$$
\end{lemma}

\begin{proof} The idea here is to use the hypothesis that the energy is large in order to find a large subset $A' \subset A$ such that $d_*(A)$ is small, and to then apply Lemmas \ref{A-A} and \ref{ShSum} to complete the proof.

We can write
$$E^*(A)=\sum_x |A \cap xA|^2.$$
Note that
$$\sum_{x: |A \cap xA| \leq \frac{E^*(A)}{2|A|^2}} |A \cap xA|^2 \leq \frac{E^*(A)}{2|A|^2} \sum_x |A\cap xA| \leq \frac{E^*(A)}{2}$$
and so
$$\sum_{x: |A \cap xA| \geq \frac{E^*(A)}{2|A|^2}} |A \cap xA|^2 \geq \frac{E^*(A)}{2}.$$
Therefore, by a dyadic pigeonholing argument, there exists $\frac{E^*(A)}{2|A|^2} \leq \tau \leq |A|$ such that
$$\sum_{x : \tau \leq |A \cap xA| < 2\tau} |A \cap xA|^2 \gtrapprox E^*(A).$$
We label $S_{\tau}:=\{x : \tau \leq |A \cap xA| < 2\tau\}$, and thus we have
\begin{equation}
|S_{\tau}|\tau^2 \gtrapprox E^*(A).
\label{dyadic}
\end{equation}
%Furthermore, since $\tau \leq |A|$, we have
%$$|S_{\tau}| \gg \frac{|A|}{K\log |A|}.$$
Observe that
\begin{equation}
\label{eq1}
\sum_{a\in A} |A \cap aS_{\tau}| = \sum_{x \in S_{\tau}}|A \cap xA| \geq |S_{\tau}|\tau.
\end{equation}
%Define 
%$$A':=\left\{a \in A :  |A \cap aS_{\tau}| \geq \frac{|S_{\tau}|\tau}{2|A|} \right\}.$$
%Then
%$$\sum_{a\in A \setminus A'} |A \cap aS_{\tau}| \leq  \frac{|S_{\tau}|\tau}{2},$$
%and thus
%$$\sum_{a\in A'} |A \cap aS_{\tau}| \geq  \frac{|S_{\tau}|\tau}{2}.$$
%It follows that
%\begin{equation}
%|A'| \geq \frac{|S_{\tau}|\tau}{2|A|} \gg \frac{|A|^{1/2}|S_{\tau}|^{1/2}}{K^{1/2} \log^{1/2}|A|}\gg \frac{|A|}{K\log %|A|}.
%\label{A'bound}
%\end{equation}
Now apply another dyadic pigeonholing argument to obtain a subset $A' \subset A$ such that for all $a \in A'$
\begin{equation}
t \leq |A \cap aS_{\tau}| <2t,
\label{dstar}
\end{equation}
for some real number $0<t \leq |A|$, and such that
$$\sum_{a \in A'} |A \cap aS_{\tau}| \gtrapprox |S_{\tau}|\tau.$$
Therefore
\begin{equation}
|A'|t \gtrapprox |S_{\tau}|\tau.
\label{dyadic2}
\end{equation}
Note, since $t \leq |A|$, that
\begin{equation}
|A'| \gtrapprox \frac{|S_{\tau}|\tau}{|A|}.
\label{dyadic3}
\end{equation}

For every $a \in A'$, we have $|A \cap aS_{\tau}| \geq t$. Therefore, we can take 
$$t=t, Q=A, R=S_{\tau}^{-1}$$
in the definition of $d_*(A')$. We then have
\begin{equation}
d_*(A') \leq \frac{|A|^2|S_{\tau}|^2}{|A'|t^3}.
\end{equation}
Apply Lemma \ref{A-A} to get
\begin{align*}
|A-A| & \geq |A'-A'| \\
& \gtrapprox \frac{|A'|^{8/5}}{d_*^{3/5}(A')} \\
& \gg |A'|^{8/5} \left( \frac{|A'|t^3}{|A|^2|S_{\tau}|^2} \right )^{3/5} \\
& \gtrapprox \frac{(|S_{\tau}|\tau)^{9/5} |A'|^{2/5}}{|A|^{6/5}|S_{\tau}|^{6/5}} \\
& \gtrapprox \frac{(|S_{\tau}|\tau)^{9/5} \left(\frac{|S_{\tau}|\tau}{|A|}\right)^{2/5}}{|A|^{6/5}|S_{\tau}|^{6/5}}  \\
& = \frac{|S_{\tau}|\tau^{11/5}}{|A|^{8/5}} \\
& \gtrapprox \frac{E^*(A)\left( \frac{E^*(A)}{|A|^2} \right)^{1/5}}{|A|^{8/5}} \\
& \geq \frac{ \left( \frac{|A|^3}{K} \right)^{6/5}}{|A|^2} = \frac{|A|^{8/5}}{K^{6/5}}.
\end{align*}

Similarly, an application of Lemma \ref{ShSum} gives

\begin{align*}
|A+A| & \geq |A'+A'| \\
& \gtrapprox \frac{|A'|^{58/37}}{d_*^{21/37}(A')} \\
& \gg |A'|^{58/37} \left( \frac{|A'|t^3}{|A|^2|S_{\tau}|^2} \right )^{21/37} \\
& \gtrapprox \frac{(|S_{\tau}|\tau)^{63/37} |A'|^{16/37}}{|A|^{42/37}|S_{\tau}|^{42/37}} \\
& \gtrapprox \frac{(|S_{\tau}|\tau)^{63/37} \left(\frac{|S_{\tau}|\tau}{|A|}\right)^{16/37}}{|A|^{42/37}|S_{\tau}|^{42/37}}  \\
& = \frac{|S_{\tau}|\tau^{79/37}}{|A|^{58/37}} \\
& \gtrapprox \frac{E^*(A)\left( \frac{E^*(A)}{|A|^2} \right)^{5/37}}{|A|^{58/37}} \\
& \geq \frac{ \left( \frac{|A|^3}{K} \right)^{42/37}}{|A|^{68/37}} = \frac{|A|^{58/37}}{K^{42/37}}.
\end{align*}

\end{proof}

We are now ready to prove the new lower bounds for $|A(A-A)|$ and $|A(A+A)|$.

\begin{theorem} \label{A(A-A)}
For any finite set $A \subset \mathbb R$,
$$|A(A-A)| \gtrapprox |A|^{\frac{3}{2}+\frac{1}{34}}.$$
\end{theorem}

\begin{proof} Write $E^*(A)=\frac{|A|^3}{K}$. The proof considers two cases:

\textbf{Case 1} - Suppose that $K \geq |A|^{1/17}$. Then apply Lemma \ref{MORNSlemma}, which tells us that
$$\frac{|A|^6}{\log |A|} \leq E^*(A)|A(A-A)|^2 = \frac{|A|^3|A(A-A)|^2}{K}$$
and therefore
$$|A(A-A)| \gtrapprox |A|^{3/2}K^{1/2} \geq |A|^{\frac{3}{2}+\frac{1}{34}}.$$

\textbf{Case 2} - Suppose that $K \leq |A|^{1/17}$. Then, by Lemma \ref{BSG2}, 
$$|A(A-A)| \geq |A-A| \gtrapprox \frac{|A|^{8/5}}{K^{6/5}} \geq |A|^{\frac{8}{5}-\frac{6}{85}}=|A|^{\frac{3}{2}+\frac{1}{34}}.$$

\end{proof}

\begin{theorem} \label{A(A+A)}
For any finite set $A \subset \mathbb R$,
$$|A(A+A)| \gtrapprox |A|^{\frac{3}{2}+\frac{5}{242}}.$$
\end{theorem}

\begin{proof} Write $E^*(A)=\frac{|A|^3}{K}$. The proof considers two cases:

\textbf{Case 1} - Suppose that $K \geq |A|^{5/121}$. Then Lemma \ref{MORNSlemma} tells us that
$$|A(A+A)| \gtrapprox |A|^{3/2}K^{1/2} \geq |A|^{\frac{3}{2}+\frac{5}{242}}.$$

\textbf{Case 2} - Suppose that $K \leq |A|^{5/121}$. Then, by Lemma \ref{BSG2}, 
$$|A(A+A)| \geq |A+A| \gtrapprox \frac{|A|^{58/37}}{K^{42/37}} \geq |A|^{\frac{58}{37}-\frac{5.42}{121.37}}=|A|^{\frac{3}{2}+\frac{5}{242}}.$$

\end{proof}

%One could also use the same method to prove the bound $|A(A+A)| \gg |A|^{\frac{3}{2}+ \frac{5}{474}}$, which %is better than the result $|A(A+A)| \gg |A|^{\frac{3}{2}+ \frac{1}{178}}$ from \cite{MORNS}. However, we do not %include the proof of this, since one can do better still by following the method of the original proof of the bound $|%A(A+A)| \gg |A|^{\frac{3}{2}+ \frac{1}{178}}$ from \cite{MORNS}, but instead utilising the following estimate %from \cite[Theorem 11]{KS} at the appropriate point in the argument:
%$$|A/A|^{10}|A+A|^{12} \gg |A|^{29}.$$
%After some calculations, the state of the art bound obtained via this method is
%$$|A(A+A)| \gg |A|^{\frac{3}{2}+\frac{1}{92}}.$$

\section{Acknowledgements}

I am grateful for the hospitality of the R\'{e}nyi Institute, where part of this work was done with support from Grant ERC-AdG. 321104. I also thank Antal Balog, Brendan Murphy, Orit Raz, Ilya Shkredov and Endre Szemer\'{e}di for helpful discussions.

\end{document}